\documentclass[11pt]{amsart}

\usepackage{amsmath,amssymb,amsthm,amsfonts}
\usepackage{mathrsfs}
\usepackage{hyperref}
\usepackage{enumerate}
\usepackage{tikz}
\usepackage{bm}
\usepackage{geometry}
\usepackage{color}
\usepackage{soul}
\newtheorem{theorem}{Theorem}[section]
\newtheorem{lemma}[theorem]{Lemma}
\newtheorem{proposition}[theorem]{Proposition}
\newtheorem{corollary}[theorem]{Corollary}

\theoremstyle{definition}

\newtheorem{assumption}[theorem]{Assumption}

\theoremstyle{remark}
\newtheorem{remark}[theorem]{Remark}

\newcommand{\supp}{supp}
\newcommand{\Dom}{Dom}
\newcommand{\Ric}{Ric}
\newcommand{\E}{E}
\newcommand{\R}{\mathbb{R}}

\title[Recovering Riemannian Geometry from Diffusion]
{Recovering Riemannian Geometry from Diffusion}

\author{Amandip Sangha}
\address{The Climate and Environmental Research Institute NILU}
\email{asan@nilu.no}

\begin{document}

\begin{abstract}
We present an intrinsic reconstruction of Riemannian geometry from a symmetric,
strongly local diffusion semigroup.
Starting from a diffusion operator and its associated first- and second-order
diffusion calculus, we recover the full weighted Riemannian structure of the
underlying manifold.
In particular, we show that the carré du champ determines a unique smooth
Riemannian metric, that the iterated carré du champ encodes curvature, and that
the symmetry of the diffusion fixes the Levi--Civita connection and reference
measure.
As a consequence, the diffusion semigroup determines the global Riemannian
manifold uniquely up to isometry.
The results provide an information-theoretic perspective on differential
geometry in which geometric structure emerges from the intrinsic behavior of
diffusion, without assuming any prior metric or coordinate description.
\end{abstract}


\maketitle
\tableofcontents

\section{Introduction}

Diffusion processes describe the evolution of information under random motion.
They arise naturally in physics, probability, and analysis, modeling phenomena
such as heat flow, particle dispersion, and stochastic transport.
Mathematically, diffusion is encoded by Markov semigroups and their generators,
which act as information channels transforming initial data into progressively
less localized states.

A classical insight of geometric analysis is that diffusion reflects the geometry
of the space on which it acts.
On a Riemannian manifold, the Laplace--Beltrami operator governs diffusion, and
geometric quantities such as distance, curvature, and volume influence how
functions evolve under the heat flow.
Traditionally, this relationship is studied by assuming a Riemannian structure
and analyzing the associated differential operator.

In this paper we reverse this perspective.
Rather than starting from a metric and studying its diffusion, we begin with a
symmetric, strongly local diffusion semigroup and ask what geometric structure
it intrinsically encodes.
Our guiding principle is that geometry should be viewed as the structure that
governs the flow of information under diffusion.
The way diffusion dissipates information at first and second order already
contains the data of a Riemannian manifold.

The central objects in our approach are the carré du champ operator and its
iterated counterpart in the sense of Bakry--\'Emery.
These operators form an intrinsic first- and second-order calculus associated
with the diffusion semigroup.
We show that this calculus is sufficient to reconstruct the full Riemannian
package: a smooth metric, its Levi--Civita connection, curvature, and the
reference measure with respect to which the diffusion is symmetric.

A key feature of the framework is that no Riemannian metric, coordinate system,
or heat kernel representation is assumed a priori.
The metric emerges canonically from the first-order diffusion calculus, while
curvature emerges from the second-order behavior.
Symmetry of the diffusion rigidly selects the Levi--Civita connection and fixes
the reference measure.
As a result, the diffusion semigroup determines the underlying manifold uniquely
up to Riemannian isometry.

Traditionally, diffusion processes on manifolds are constructed \emph{from} an
underlying geometric structure: a Riemannian metric determines a Laplace--Beltrami
operator, which in turn generates a heat semigroup.
The present work reverses this ontological direction.
We begin with a primitive diffusion semigroup $(P_t)_{t\ge0}$, assumed only to
satisfy minimal analytic and locality properties, and show that the full
Riemannian geometry of the underlying space can be reconstructed from this
diffusion data alone.
In this sense, geometry is not an a priori input but an emergent structure
encoded in the infinitesimal behavior of the diffusion.

More precisely, the reconstruction proceeds through the Bakry--\'Emery
$\Gamma$--calculus.
Starting from the semigroup $(P_t)$, we extract the carré du champ operator
$\Gamma$ as the first-order infinitesimal covariance of the diffusion.
This bilinear form canonically defines a smooth co-metric on the cotangent
bundle, whose inverse yields a Riemannian metric $g$.
Higher-order diffusion data encoded in $\Gamma_2$ then recover curvature, the
Levi--Civita connection, and the natural reference measure, culminating in the
identification of the generator as a weighted Laplace--Beltrami operator.
Thus the logical chain
\[
\text{diffusion semigroup } P_t
\;\longrightarrow\;
\text{Bakry--\'Emery calculus } (\Gamma,\Gamma_2)
\;\longrightarrow\;
\text{Riemannian geometry } (g,\nabla,\mu)
\]
is made fully explicit and rigorous.

This work forms part of a broader information-theoretic approach to geometry.
In earlier work, smooth manifold structure itself was shown to be recoverable
from entropy and information-theoretic probes of diffusion
\cite{SanghaEntropySmooth}.
Here we build on that foundation and show that, once smoothness is available,
the full Riemannian geometry is already encoded in the intrinsic behavior of
diffusion.
The resulting viewpoint places diffusion calculus, rather than coordinates,
spectral data, or variational principles, at the foundation of differential
geometry.

\section{Preliminaries}

\subsection{Diffusion semigroup}\label{subsec:diffusion-semigroup}

Throughout we fix a smooth, connected manifold $M$ and a $\sigma$--finite Borel measure $\mu$
with full support.\footnote{Full support is used only to avoid degenerate quotients; it can be
dropped at the expense of working on $\mathrm{supp}(\mu)$.}

\medskip
\noindent\textbf{Standing diffusion assumptions.}
We are given a conservative symmetric Markov $C_0$--semigroup $(P_t)_{t\ge 0}$ on $L^2(M,\mu)$.
Let $L$ denote its (self-adjoint) generator:
\[
Lf := \lim_{t\downarrow 0}\frac{P_tf-f}{t},\qquad f\in \mathrm{Dom}(L)\subset L^2(M,\mu).
\]
We write $\mathcal E$ for the (closed) symmetric Dirichlet form associated to $(P_t)$, defined by
\[
\mathcal{E}(f,g):=\lim_{t\downarrow 0}\frac1t\,\langle f-P_tf,\; g\rangle_{L^2(\mu)},
\]
whenever the limit exists and is finite; we set $\mathcal D(\mathcal E):=\{f\in L^2(\mu):\mathcal E(f,f)<\infty\}$.
Equivalently, for $f\in\mathrm{Dom}(L)$ and $g\in\mathcal D(\mathcal E)$ one has
$\mathcal E(f,g)=-\langle Lf,g\rangle_{L^2(\mu)}$, and also
$\mathcal E(f,g)=\langle(-L)^{1/2}f,(-L)^{1/2}g\rangle_{L^2(\mu)}$.

\begin{assumption}[Dirichlet form / diffusion axioms]\label{ass:D}
We assume:
\begin{enumerate}
\item[(D1)] (\emph{Markov and conservative}) $0\le f\le 1 \Rightarrow 0\le P_tf\le 1$ and $P_t1=1$.
\item[(D2)] (\emph{Symmetry}) $\int_M (P_tf)\,g\,d\mu=\int_M f\,(P_tg)\,d\mu$ for all $f,g\in L^2(M,\mu)$.
\item[(D3)] (\emph{Strong locality}) the form $\mathcal{E}$ is strongly local, i.e.\ $\mathcal{E}(f,g)=0$
whenever $f$ is constant on a neighbourhood of $\supp(g)$.
\item[(D4)] (\emph{Regularity}) $\mathcal{E}$ is regular in the Dirichlet-form sense: $\mathcal{D}(\mathcal{E})\cap C_c(M)$
is dense both in $\mathcal{D}(\mathcal{E})$ (for the $\mathcal{E}_1$--norm) and in $C_c(M)$ (for $\|\cdot\|_\infty$).
\item[(D5)] (\emph{Smoothness of the invariant measure}) The reference measure $\mu$ is assumed to be a smooth positive measure,
i.e.\ locally of the form $\mu=\rho\,dx$ with $\rho\in C^\infty$, $\rho>0$.
\end{enumerate}
\end{assumption}

Carré du champ as infinitesimal decorrelation.
For smooth functions $f,g\in C^\infty(M)$ we define the carré du champ intrinsically
from the diffusion semigroup by
\begin{equation}\label{eq:Gamma-semigroup}
\Gamma(f,g)
:=
\lim_{t\downarrow 0}\frac{1}{2t}
\Big(
P_t(fg) - (P_t f)(P_t g)
\Big),
\end{equation}
whenever the limit exists.
We write $\Gamma(f):=\Gamma(f,f)$.

Because the paper aims to reconstruct \emph{Riemannian} objects, we need a smooth calculus for $\Gamma$
on a smooth core; we do \emph{not} assume a metric at this stage.

\begin{assumption}[Smooth first-order diffusion calculus]
For all $f,g\in C^\infty(M)$, the limit \eqref{eq:Gamma-semigroup} exists pointwise
and defines a function $\Gamma(f,g)\in C^\infty(M)$.
Moreover:
\begin{itemize}
\item $\Gamma$ is symmetric and bilinear;
\item $\Gamma(f,g)(x)$ depends only on the first-order jets $df(x),dg(x)$;
\item $\Gamma$ is local: if $g$ is constant near $\supp(f)$ then $\Gamma(f,g)=0$;
\item $\Gamma$ satisfies the Leibniz  rule pointwise on $M$, i.e. for all $f,g,h\in C^\infty(M)$,
\[
\Gamma(fg,h)=f\,\Gamma(g,h)+g\,\Gamma(f,h),
\qquad
\Gamma(h,fg)=f\,\Gamma(h,g)+g\,\Gamma(h,f).
\]

\item (One-variable chain rule) For every smooth $\phi:\mathbb{R}\to\mathbb{R}$,
\[
\Gamma(\phi\circ f,g)=\phi'(f)\,\Gamma(f,g)
\quad\text{and}\quad
\Gamma(f,\phi\circ g)=\phi'(g)\,\Gamma(f,g).
\]
\end{itemize}
\end{assumption}
\begin{remark}
Assumption~2.2 is a smoothness hypothesis on the diffusion.
It holds, for example, when $L$ is a smooth second--order diffusion operator
and $\mathcal A\subset C^\infty(M)$ is a dense algebraic core, invariant under
$(P_t)$, closed under products and composition, and such that
$\mathcal A\subset\Dom(L)$ and $L\mathcal A\subset\mathcal A$.
All subsequent identities are understood on this fixed smooth core.
\end{remark}

All subsequent geometric reconstructions (metric, connection, curvature, measure)
will be derived solely from the diffusion semigroup, the Dirichlet form, and the
first-order calculus encoded by $\Gamma$; no local coordinate or heat-kernel
representation of $L$ is assumed.

\begin{lemma}[Multivariate chain rule]\label{lem:multichain}
Assume the one-variable chain rule in Assumption~2.2, together with bilinearity and
the Leibniz rule.
Then for every smooth $\Phi:\mathbb{R}^m\to\mathbb{R}$ and every
$F=(f_1,\dots,f_m)\in C^\infty(M;\mathbb{R}^m)$,
\[
\Gamma(\Phi\circ F,\,g)
=
\sum_{a=1}^m (\partial_a\Phi)(F)\,\Gamma(f_a,g),
\qquad g\in C^\infty(M).
\]
\end{lemma}
\begin{proof}
Fix $x\in M$ and $g\in C^\infty(M)$. Write $F=(f_1,\dots,f_m)$ and set
\[
u := F(x) = \bigl(f_1(x),\dots,f_m(x)\bigr)\in\R^m.
\]
Let
\[
\ell_x(y)
:= \Phi(u) + \sum_{a=1}^m (\partial_a\Phi)(u)\,\bigl(f_a(y)-f_a(x)\bigr),
\qquad y\in M,
\]
i.e. $\ell_x$ is the first-order Taylor polynomial of $\Phi\circ F$ at $x$ expressed in terms of the functions $f_a$.

Define the remainder
\[
R_x := \Phi\circ F - \ell_x.
\]
In local coordinates near $x$, Taylor's theorem for $\Phi:\R^m\to\R$ gives, for $y$ near $x$,
\begin{align}
\Phi\bigl(F(y)\bigr)
&= \Phi(u) + \sum_{a=1}^m (\partial_a\Phi)(u)\,\bigl(f_a(y)-f_a(x)\bigr)
+ \sum_{a,b=1}^m \bigl(f_a(y)-f_a(x)\bigr)\bigl(f_b(y)-f_b(x)\bigr)\,H_{ab}(y),
\label{eq:taylor-mv}
\end{align}
where each $H_{ab}$ is a smooth function (one may take
$H_{ab}(y)=\int_0^1 (1-s)\,(\partial_{ab}\Phi)\bigl(u+s(F(y)-u)\bigr)\,ds$).
Comparing with the definition of $\ell_x$ shows
\[
R_x(y) = \sum_{a,b=1}^m \bigl(f_a(y)-f_a(x)\bigr)\bigl(f_b(y)-f_b(x)\bigr)\,H_{ab}(y).
\]
In particular,
\[
R_x(x)=0,
\qquad
dR_x(x)=0,
\]
so $R_x$ has vanishing $1$-jet at $x$.

By Assumption~2.2, $\Gamma(\cdot,\cdot)(x)$ depends only on first-order jets.
Since $R_x$ and the constant function $0$ have the same $1$-jet at $x$, we have
\[
\Gamma(R_x,g)(x)=\Gamma(0,g)(x)=0,
\]
hence
\[
\Gamma(\Phi\circ F,g)(x)=\Gamma(\ell_x,g)(x).
\]

First, constants have zero carr\'e du champ with anything: letting $\mathbf 1$ be the constant function $1$,
the Leibniz rule gives
\[
\Gamma(\mathbf 1,g)=\Gamma(\mathbf 1\cdot \mathbf 1,g)
= \mathbf 1\,\Gamma(\mathbf 1,g)+\mathbf 1\,\Gamma(\mathbf 1,g)
= 2\,\Gamma(\mathbf 1,g),
\]
so $\Gamma(\mathbf 1,g)=0$, and thus $\Gamma(c,g)=\Gamma(c\mathbf 1,g)=c\,\Gamma(\mathbf 1,g)=0$ for all constants $c$.
Therefore,
\begin{align*}
\Gamma(\ell_x,g)(x)
&= \Gamma\!\left(\Phi(u) + \sum_{a=1}^m (\partial_a\Phi)(u)\,\bigl(f_a-f_a(x)\bigr),\, g\right)\!(x) \\
&= \sum_{a=1}^m (\partial_a\Phi)(u)\,\Gamma\bigl(f_a-f_a(x),\,g\bigr)(x)
\qquad \text{(bilinearity, and $\Gamma(\Phi(u),g)=0$)} \\
&= \sum_{a=1}^m (\partial_a\Phi)(u)\,\Gamma(f_a,g)(x),
\end{align*}
since $\Gamma(f_a(x),g)=0$ for the constant function $f_a(x)$.

\medskip
Combining the steps yields, for every $x\in M$,
\[
\Gamma(\Phi\circ F,g)(x)=\sum_{a=1}^m (\partial_a\Phi)\bigl(F(x)\bigr)\,\Gamma(f_a,g)(x).
\]
As $x$ was arbitrary, this holds pointwise on $M$.
\end{proof}

\begin{proposition}[Equivalence with generator formula]
If $f,g\in C^\infty(M)$ belong to $\mathrm{Dom}(L)$ and $fg\in\mathrm{Dom}(L)$, then
the carré du champ defined by \eqref{eq:Gamma-semigroup} satisfies
\[
\Gamma(f,g)
=
\frac12\big(L(fg)-f\,Lg-g\,Lf\big).
\]    
\end{proposition} 
\begin{proof}
Fix $f,g\in C^\infty(M)\cap\Dom(L)$ with $fg\in\Dom(L)$. For $t>0$ set
\[
A_t := P_t(fg) - (P_tf)(P_tg).
\]
Add and subtract $fg$, $fP_tg$, and $gP_tf$ to obtain the exact algebraic identity
\begin{align}
A_t
&= \bigl(P_t(fg)-fg\bigr) - f\bigl(P_tg-g\bigr) - g\bigl(P_tf-f\bigr) - \bigl(P_tf-f\bigr)\bigl(P_tg-g\bigr).
\label{eq:At-decomp}
\end{align}
Divide by $2t$:
\begin{align}
\frac{1}{2t}A_t
&= \frac12\,\frac{P_t(fg)-fg}{t}
-\frac12\,f\,\frac{P_tg-g}{t}
-\frac12\,g\,\frac{P_tf-f}{t}
-\frac12\,\frac{(P_tf-f)(P_tg-g)}{t}.
\label{eq:At-over-2t}
\end{align}

Since $h\in\Dom(L)$ iff $(P_th-h)/t\to Lh$ in $L^2(\mu)$ as $t\downarrow0$, the first three terms on the right of \eqref{eq:At-over-2t} converge in $L^2(\mu)$ to
\[
\frac12\,L(fg) - \frac12\, f\,Lg - \frac12\, g\,Lf.
\]
It remains to show that the last term in \eqref{eq:At-over-2t} tends to $0$ in $L^1(\mu)$ (hence in measure). Using Cauchy--Schwarz and the $L^2$-generator expansions,
\begin{align*}
\left\|\frac{(P_tf-f)(P_tg-g)}{t}\right\|_{L^1(\mu)}
&\le \frac{1}{t}\,\|P_tf-f\|_{L^2(\mu)}\,\|P_tg-g\|_{L^2(\mu)} \\
&= \frac{1}{t}\,\bigl(t\|Lf\|_{L^2(\mu)}+o(t)\bigr)\,\bigl(t\|Lg\|_{L^2(\mu)}+o(t)\bigr) \\
&= t\,\|Lf\|_{2}\|Lg\|_{2} + o(t)\xrightarrow[t\downarrow0]{}0.
\end{align*}
Therefore
\[
\frac{1}{2t}\Bigl(P_t(fg)-(P_tf)(P_tg)\Bigr)
\;\longrightarrow\;
\frac12\Bigl(L(fg)-f\,Lg-g\,Lf\Bigr)
\quad\text{in }L^1(\mu).
\]
By definition \eqref{eq:At-over-2t} of $\Gamma$ via the semigroup limit, the left-hand side converges pointwise to $\Gamma(f,g)$, and by Assumption~2.2 this limit is a smooth function. Hence the above $L^1$ convergence forces
\[
\Gamma(f,g)=\frac12\Bigl(L(fg)-f\,Lg-g\,Lf\Bigr)
\quad\mu\text{-a.e.}
\]
Since both sides are smooth functions on $M$, equality holds everywhere on $M$.
\end{proof}

\subsection{Information-theoretic quantities}
The information-theoretic quantities introduced in this subsection provide
motivation and connections to earlier work.
They are not used technically in the intrinsic geometric reconstruction
carried out in Sections 3 and 4.

Let $(P_t)_{t\ge 0}$ denote the symmetric, strongly local diffusion semigroup on $M$ introduced
above. For a probability measure $\mu$ on $M$, we write $P_t \mu$ for its image under the semigroup, $P_t\mu(A) = \int_{A}P_t 1_{A}(x)d\mu(x).$
All information-theoretic quantities used in this paper are defined in terms of such distributions
and do not require any underlying metric.

\medskip
\noindent\textbf{Entropy.}
If $\rho$ is a probability density with respect to the reference measure $\mu$, we define
\[
H(\rho) := - \int_M \rho \log \rho \, d\mu.
\]
If $X$ is a random variable on $M$ with density $\rho$, we write $H(X):=H(\rho)$.

\medskip
\noindent\textbf{Conditional entropy.}
Let $Y$ be an $M$--valued random variable and $L\in\{0,1\}$ a binary random 
variable. We write $p_{L,Y}(l,y)$ for the joint density of $(L,Y)$ with 
respect to the product measure $\lambda \otimes \mu$, where $\lambda$ is 
counting measure on $\{0,1\}$. The conditional entropy of $L$ given $Y$ is
\[
H(L \mid Y)
:= - \sum_{\ell\in\{0,1\}} \int_M p_{L,Y}(\ell,y)\,
   \log p_{L \mid Y}(\ell \mid y)\, d\mu(y),
\]
where $p_{L\mid Y}$ is the conditional density.

\medskip
\noindent\textbf{Relative entropy.}
For two probability densities $\rho$ and $\sigma$ with respect to $\mu$, the relative entropy
(Kullback--Leibler divergence) \cite{CoverThomas2006} is
\[
D(\rho \,\|\, \sigma)
:= \int_M \rho \log\!\left(\frac{\rho}{\sigma}\right) d\mu.
\]
More generally, if $\nu$ and $\eta$ are probability measures with $\nu \ll \eta$, we write
\[
D(\nu \,\|\, \eta)
:= \int_M \log\!\left(\frac{d\nu}{d\eta}\right) d\nu.
\]
In our applications, relative entropy will be evaluated between measures obtained from $(P_t)$ and
suitable reference measures.

\medskip
\noindent\textbf{Mutual information.}
If $(U,V)$ is a pair of random variables with joint density $\rho_{U,V}$ and marginals
$\rho_U and \rho_V$, the mutual information \cite{CoverThomas2006} is
\[
I(U;V)
= H(U) - H(U\mid V)
= D\big(\rho_{U,V} \,\big\|\, \rho_U \otimes \rho_V\big).
\]
Monotonicity of $I(U;V)$ under the action of $(P_t)$ (data-processing inequality) underlies the
entropy dissipation identities used in \cite{Sangha2025Isoperimetry}.

\medskip

In later sections we will specialise these general notions to distributions obtained by restricting
the diffusion to two-dimensional slices and comparing them to flat reference diffusions.

\subsection{The noisy-label model}
\label{subsec:noisy-label-model}
A central information-theoretic object in our framework is the binary noisy-label experiment
introduced in \cite{Sangha2025Isoperimetry}. Let $E\subset M$ be a measurable set with indicator
function $1_E$. We define the binary random variable
\[
B := 1_E(X),
\]
where $X$ is an $M$--valued random variable with law $\mu$, $P(X\in A)=\mu(A)$. The random variable $B\in\{0,1\}$ is observed
only after diffusion: for $t>0$ we define $Y_t := X_t$, where $X_t$ has law $P_t\mu$.

The joint distribution of $(B,Y_t)$ is
\[
\mathbb{P}(B=\beta,\, Y_t\in A)
= \int_{E_\beta} P_t(x,A)\, d\mu(x),
\qquad \beta\in\{0,1\},
\]
where $E_1=E$ and $E_0=E^c = M\setminus E$. Moreover,
\[
\mathbb{P}(B=\beta \mid Y_t \in A)
= \frac{1}{\mu(E_\beta)} \int_A P_t 1_{E_\beta}(y)\, d\mu(y),
\qquad A\subset M\ \text{measurable}.
\]

The conditional entropy $H(B\mid Y_t)$ measures the uncertainty in the label after observing
the diffused point $Y_t$. 

\subsection{Mutual information under diffusion flow}

Let $X$ be an $M$--valued random variable with law $\mu$, and let $X_t$ have law $P_t\mu$.
For any random variable $U$ (for instance the binary label $B=1_E(X)$), the pair $(U,X_t)$
has joint density $p_{U,X_t}$ and marginals $p_U$, $p_{X_t}$. The mutual information is
\[
I(U;X_t)
= D\big(p_{U,X_t} \,\big\|\, p_U \otimes p_{X_t}\big)
= H(U) - H(U\mid X_t).
\]

Because $X \mapsto X_t$ is a Markov channel generated by the diffusion semigroup $(P_t)$,
the data--processing inequality \cite{CoverThomas2006} gives
\begin{equation}\label{eq:MI-monotone}
I(U;X_t) \le I(U;X),
\qquad t \ge 0,
\end{equation}
for every choice of $U$. Thus mutual information is nonincreasing along the diffusion.

In the binary noisy--label case $U=B=1_E(X)$ from Section~3.1, we have
\[
I(B;X_t) = H(B) - H(B\mid X_t).
\]
The small--time behaviour of $H(B\mid X_t)$ will be analysed later in the paper by an
information dissipation identity, following \cite{Sangha2025Isoperimetry}.

\subsection{Information dissipation identity}
Let $B = 1_E(X)$ be the binary random variable from \ref{subsec:noisy-label-model} and let $Y_t := X_t$ have law $P_t\mu$.
The conditional entropy satisfies
\[
H(B\mid Y_t)
= - \sum_{\beta\in\{0,1\}} \int_M p_{B,Y_t}(\beta,y)
    \log p_{B\mid Y_t}(\beta\mid y)\, d\mu(y).
\]

Since $p_{Y_t\mid B=\beta}(y) = P_t1_{E_\beta}(y)/\mu(E_\beta)$, we have
\[
p_{B\mid Y_t}(1\mid y)
= \frac{P_t 1_E(y)}{P_t 1_E(y) + P_t 1_{E^c}(y)}
= P_t 1_E(y),
\]
because $P_t 1_E + P_t 1_{E^c} = P_t 1 = 1$.

Hence
\[
H(B\mid Y_t)
= -\int_M\!\left[
P_t 1_E(y)\log P_t 1_E(y) + (1-P_t 1_E(y))\log(1-P_t 1_E(y))
\right] d\mu(y).
\]

Differentiating under the integral and using $L1=0$ gives the
\emph{information dissipation identity} \cite{Sangha2025Isoperimetry},
\begin{equation}\label{eq:ID-identity}
\frac{d}{dt}H(B\mid Y_t)
= - \int_M 
     \frac{\Gamma(P_t 1_E)(y)}
          {P_t 1_E(y)\, (1 - P_t 1_E(y))}
     \, d\mu(y),
\end{equation}
where $\Gamma$ is the Carré du champ operator associated with the generator $L$.

Since $P_t1_E$ is $C^\infty$ for every $t>0$ \cite{BakryGentilLedoux2014, FukushimaOshimaTakeda2011} and satisfies $0<P_t1_E<1$ by the strong maximum
principle for symmetric diffusion semigroups \cite{FriedmanParabolicPDE}, the denominator in
\eqref{eq:ID-identity} is strictly positive everywhere and the integrand is smooth. Hence the
integral in \eqref{eq:ID-identity} is finite for all $t>0$.

Identity \eqref{eq:ID-identity} expresses the first-order loss of label information entirely in
terms of $\Gamma$ and the smoothed indicator $P_t 1_E$.

\subsection{Smooth structure from information} 
In this paper we assume that $M$ carries a smooth structure, so that the
diffusion semigroup $(P_t)$ acts on $C^\infty(M)$; no heat kernel asymptotics are assumed unless explicitly stated. However, smoothness itself can be recovered from
purely information-theoretic data.  In our preceding companion paper \cite{SanghaEntropySmooth},
the small--scale quadratic entropy response of local information probes
determines a unique $C^\infty$--structure on any topological manifold.
Entropy coordinate charts are shown to be classical smooth charts, and the
entropy--smooth functions coincide with the ordinary smooth functions.

Thus the present paper may be viewed as the second layer of a unified
information--theoretic reconstruction program: the smooth structure of $M$
may be recovered from local entropy behaviour, and on top of this layer the
Riemannian metric, Levi--Civita connection, curvature operator, and global
isometry class of $(M,g)$ may be recovered from the small--time information
flow of the diffusion semigroup.

\section{Recovering the Riemannian Metric}

\subsection{The intrinsic distance associated with diffusion}

The carré du champ $\Gamma$ induces an intrinsic distance on $M$ defined by
the dual variational formula
\begin{equation}\label{eq:dual-dGamma}
d_\Gamma(x,y)
:=
\sup\Bigl\{
u(x)-u(y):
u\in C_c^\infty(M),\ \Gamma(u,u)\le 1\ \text{a.e.}
\Bigr\}.
\end{equation}
This is the standard intrinsic metric associated with strongly local Dirichlet
forms \cite{FukushimaOshimaTakeda2011}.

\section{Recovering the Riemannian Structure from Information Diffusion}

\subsection{Recovering the metric tensor}\label{subsec:recover-metric}

In this subsection we show that the carré du champ $\Gamma$ determines a unique smooth
Riemannian metric on $M$.  Throughout we assume only Assumptions~2.1--2.2.

\medskip
\noindent\textbf{Well-defined bilinear form on cotangent vectors.}
Fix $x\in M$.  By Assumption~2.2, for $f,g\in C^\infty(M)$ the value $\Gamma(f,g)(x)$ depends
only on the differentials $df(x),dg(x)$.  Hence we may define a symmetric bilinear form
\begin{equation}\label{eq:cometric-def}
\langle \alpha,\beta\rangle_{g_x^{-1}}
\;:=\;
\Gamma(f,g)(x),
\end{equation}
where $f,g\in C^\infty(M)$ are any smooth functions satisfying $df(x)=\alpha$ and $dg(x)=\beta$.
This is well-defined: if $f_1,f_2$ satisfy $df_1(x)=df_2(x)$ then $\Gamma(f_1-f_2,h)(x)=0$
for all $h$, hence $\Gamma(f_1,h)(x)=\Gamma(f_2,h)(x)$, and similarly in the second entry.
Bilinearity and symmetry follow from the bilinearity and symmetry of $\Gamma$.

\medskip
\noindent\textbf{Local coordinate description.}
Let $(U;\,x^1,\dots,x^n)$ be a coordinate chart and write, on $U$,
\begin{equation}\label{eq:Gij}
G^{ij} := \Gamma(x^i,x^j).
\end{equation}
Since $\Gamma$ is smooth on smooth inputs (Assumption~2.2), the functions $G^{ij}$ are smooth
on $U$.  Moreover, for any $f\in C^\infty(U)$ the chain rule in Assumption~2.2 gives
\begin{equation}\label{eq:Gamma-symbol}
\Gamma(f,g)
=
\sum_{i,j=1}^n \,(\partial_i f)\,(\partial_j g)\,G^{ij}
\qquad\text{on }U,
\end{equation}
by writing $f=F(x^1,\dots,x^n)$ and $g=H(x^1,\dots,x^n)$ locally and applying the chain rule
with $\Phi=F$ and $\Phi=H$ componentwise.
Thus $\Gamma$ is a first-order bilinear differential operator whose ``principal part'' in
coordinates is given by the symmetric matrix $(G^{ij})$.

\medskip
\noindent\textbf{Non-degeneracy and inversion.}
To obtain a Riemannian (rather than sub-Riemannian) structure we impose the non-degeneracy
condition that for every $x\in M$ the bilinear form \eqref{eq:cometric-def} on $T_x^*M$
is positive definite.  Equivalently, in any coordinate chart the matrix $(G^{ij}(x))$ is
positive definite.  Under this assumption, we may define on $U$ the inverse matrix
\begin{equation}\label{eq:gij}
(g_{ij}) := (G^{ij})^{-1}.
\end{equation}
Since inversion is smooth on the cone of positive definite matrices, the coefficients
$g_{ij}$ are smooth on $U$.

\medskip
\noindent\textbf{Coordinate invariance.}
If $(y^1,\dots,y^n)$ is another coordinate system on $U$, then by the chain rule,
\[
\Gamma(y^a,y^b)
=
\sum_{i,j} (\partial_i y^a)(\partial_j y^b)\,\Gamma(x^i,x^j)
=
\sum_{i,j} (\partial_i y^a)(\partial_j y^b)\,G^{ij}.
\]
Thus $(G^{ab}_y)$ transforms as a $(2,0)$-tensor (a co-metric), and therefore its inverse
$(g_{ab})$ transforms as a $(0,2)$-tensor.  Hence the locally defined coefficients
$g_{ij}$ patch to a globally defined smooth symmetric $2$-tensor $g$ on $M$.

\begin{theorem}[Metric recovered from $\Gamma$]\label{thm:recover-metric-from-Gamma}
Assume in addition that the bilinear form \eqref{eq:cometric-def} is positive definite at every
$x\in M$.  Then there exists a unique smooth Riemannian metric $g$ on $M$ such that for all
$f,g\in C^\infty(M)$,
\begin{equation}\label{eq:Gamma-metric}
\Gamma(f,g)=\langle df,dg\rangle_{g^{-1}}.
\end{equation}
Equivalently, in local coordinates $(x^i)$ one has $\Gamma(x^i,x^j)=g^{ij}$.
\end{theorem}
\begin{proof}
Fix $x\in M$. By Assumption~2.2, for any $f,g\in C^\infty(M)$ the value $\Gamma(f,g)(x)$ depends
only on the differentials $df(x),dg(x)\in T_x^*M$. Hence the prescription
\[
\langle \alpha,\beta\rangle_{g^{-1}_x}:=\Gamma(f,g)(x),
\qquad df(x)=\alpha,\ dg(x)=\beta,
\]
defines a \emph{well-defined} symmetric bilinear form on $T_x^*M$ (cf.\ (5)).
By hypothesis, this bilinear form is positive definite for every $x$, so $g^{-1}_x$ is an inner product
on $T_x^*M$.

Let $(U;x^1,\dots,x^n)$ be a chart and define on $U$
\[
G^{ij}:=\Gamma(x^i,x^j).
\]
Since $\Gamma$ is smooth on smooth inputs (Assumption~2.2), each $G^{ij}\in C^\infty(U)$.
Moreover, for any $f,h\in C^\infty(U)$, writing $f=F(x^1,\dots,x^n)$ and $h=H(x^1,\dots,x^n)$
locally and applying the multivariate chain rule (Lemma~2.3) in each slot yields the coordinate formula (7):
\begin{equation}\label{eq:Gamma-coords}
\Gamma(f,h)=\sum_{i,j=1}^n (\partial_i f)(\partial_j h)\,G^{ij}
\qquad\text{on }U.
\end{equation}
Evaluating \eqref{eq:Gamma-coords} at $x\in U$ and taking $f=x^i$, $h=x^j$ gives
\[
\langle dx^i,dx^j\rangle_{g^{-1}} = \Gamma(x^i,x^j)=G^{ij},
\]
so $(G^{ij}(x))$ is the matrix of $g^{-1}_x$ in the coframe $(dx^i)_x$. Positivity of $g^{-1}_x$
is equivalent to $(G^{ij}(x))$ being symmetric positive definite, hence invertible. Define its inverse
\[
(g_{ij}) := (G^{ij})^{-1}.
\]
Since inversion is smooth on the cone of positive definite matrices, $g^{ij}\in C^\infty(U)$.

Let $(V;y^1,\dots,y^n)$ be another chart with $U\cap V\neq\varnothing$. On $U\cap V$, apply
Lemma~2.3 to each coordinate function $y^a=y^a(x^1,\dots,x^n)$ and use bilinearity of $\Gamma$:
\begin{align}
G^{(y)}_{ab}
:=\Gamma(y^a,y^b)
&=\sum_{i,j=1}^n (\partial_i y^a)(\partial_j y^b)\,\Gamma(x^i,x^j)
=\sum_{i,j=1}^n (\partial_i y^a)(\partial_j y^b)\,G^{(x)}_{ij}.
\label{eq:cometric-transform}
\end{align}
Thus the local matrices $(G^{ij})$ transform on overlaps as the components of a smooth $(2,0)$-tensor,
i.e.\ a globally defined smooth co-metric $g^{-1}\in C^\infty(M;\,S^2TM)$. Since each $g^{-1}_x$ is
positive definite, the pointwise inverse defines a smooth $(0,2)$-tensor
\[
g := (g^{-1})^{-1}\in C^\infty(M;\,S^2T^*M),
\]
whose local coefficients in the chart $(x^i)$ are exactly $(g_{ij})=(G^{ij})^{-1}$, and by
\eqref{eq:cometric-transform} these coefficients patch to a global tensor. Positivity implies $g$ is a
Riemannian metric.

Fix $x\in U$ and write $df(x)=(\partial_i f)(x)\,dx^i_x$ and $dh(x)=(\partial_j h)(x)\,dx^j_x$.
Then, since $g^{-1}(dx^i,dx^j)=G^{ij}$,
\[
\langle df,dh\rangle_{g^{-1}}(x)
=\sum_{i,j=1}^n (\partial_i f)(x)(\partial_j h)(x)\,G^{ij}(x)
=\Gamma(f,h)(x)
\]
by \eqref{eq:Gamma-coords}. Hence (9) holds pointwise on $M$.

Suppose $\tilde g$ is another smooth Riemannian metric with
\[
\Gamma(f,h)=\langle df,dh\rangle_{\tilde g^{-1}}
\qquad\forall f,h\in C^\infty(M).
\]
Fix $x\in M$ and $\alpha,\beta\in T_x^*M$. Choose a chart $(U;x^1,\dots,x^n)$ about $x$ and write
$\alpha=\alpha_i\,dx^i_x$, $\beta=\beta_j\,dx^j_x$. Let $\chi\in C_c^\infty(U)$ satisfy $\chi\equiv 1$
near $x$, and define
\[
f:=\chi\sum_{i=1}^n \alpha_i x^i,\qquad h:=\chi\sum_{j=1}^n \beta_j x^j.
\]
Then $df(x)=\alpha$ and $dh(x)=\beta$. Therefore
\[
\langle \alpha,\beta\rangle_{\tilde g^{-1}_x}
=\langle df(x),dh(x)\rangle_{\tilde g^{-1}_x}
=\Gamma(f,h)(x)
=\langle df(x),dh(x)\rangle_{g^{-1}_x}
=\langle \alpha,\beta\rangle_{g^{-1}_x}.
\]
Since $\alpha,\beta$ were arbitrary, $\tilde g^{-1}_x=g^{-1}_x$ for all $x$, hence $\tilde g=g$.
\end{proof}

\begin{remark}
At this stage no geodesics, exponential map, heat kernel asymptotics, or curvature notions
have been used.  The metric is recovered purely from the first-order diffusion calculus
encoded by $\Gamma$.
\end{remark}

\subsection{Curvature information from $\Gamma_2$}
\label{subsec:recover-curvature}

In this subsection we show that the curvature of the recovered Riemannian metric is encoded
intrinsically in the second-order diffusion calculus, via the iterated carré du champ.
No heat kernel asymptotics, small-time expansions, or directional entropy defects are used.

\medskip
\noindent\textbf{Bochner identity.}
Once the Riemannian metric $g$ has been recovered from $\Gamma$ (Section 4.1),
and the reference measure $\mu=\rho\,\mathrm{vol}_g$ has been recovered
(Section 4.4), the generator takes the form
\[
L=\Delta_g+\langle\nabla\log\rho,\nabla\cdot\rangle_g.
\]
In this setting, the Bochner identity reads
\begin{equation}\label{eq:Bochner-BE}
\Gamma_2(f)
=
\|\nabla^2 f\|_g^2
+
\big(\mathrm{Ric}_g+\nabla^2\log\rho\big)(\nabla f,\nabla f),
\end{equation}
where $\nabla^2\log\rho$ is the Hessian with respect to the Levi--Civita connection.

\begin{remark}[Generator representation of $\Gamma_2$]
In the presence of a generator $L$ and sufficient domain regularity, it is standard in the
Bakry--\'Emery calculus to define the iterated carr\'e du champ by
\[
\Gamma_2(f) := \tfrac12\,L\Gamma(f) - \Gamma(f,Lf),
\qquad f\in\mathcal A,
\]
where $\mathcal A$ is the smooth core from Assumption~2.2.
When $f,\Gamma(f)\in\Dom(L)$, this expression coincides with the semigroup representation
\[
\Gamma_2(f)
= \lim_{t\downarrow0}\frac{1}{2t}\Bigl(P_t\Gamma(f)-\Gamma(P_tf)\Bigr),
\]
see e.g.\ \cite[Chapter~3]{BakryGentilLedoux2014}.
In the present work we take the generator identity as the defining relation.
\end{remark}

Fix $x\in M$ and $v\in T_xM$.
Choose $f\in C^\infty(M)$ such that
\[
\nabla f(x)=v,
\qquad
\nabla^2 f(x)=0.
\]
Evaluating \eqref{eq:Bochner-BE} at $x$ gives
\[
\Gamma_2(f)(x)
=
\big(\mathrm{Ric}_g+\nabla^2\log\rho\big)(v,v).
\]
Thus the \emph{Bakry--\'Emery Ricci tensor} of the weighted manifold $(M,g,\mu)$,
\[
\mathrm{Ric}_\mu:=\mathrm{Ric}_g+\nabla^2\log\rho,
\]
is recovered pointwise by
\begin{equation}\label{eq:Ricmu-recovery}
\mathrm{Ric}_\mu(v,v)=\Gamma_2(f)(x),
\qquad
\nabla f(x)=v,\ \nabla^2 f(x)=0.
\end{equation}
Polarization in $v$ yields the full symmetric bilinear form $\mathrm{Ric}_\mu$.
In the special case $\rho\equiv\mathrm{const}$, one has $\mathrm{Ric}_\mu=\mathrm{Ric}_g$.

\medskip
\noindent\textbf{From Ricci to full curvature.}
The Ricci tensor, together with the Levi--Civita connection (recovered in the next subsection),
determines the full Riemann curvature tensor.
Equivalently, once the covariant Hessian $\nabla^2$ is available, the Riemann curvature operator
$R$ is recovered from the commutator identity
\[
(\nabla_X\nabla_Y - \nabla_Y\nabla_X - \nabla_{[X,Y]})\,\nabla f
=
R(X,Y)\nabla f,
\qquad
f\in C^\infty(M),
\]
which is intrinsic and involves only the connection.

\begin{theorem}[Levi--Civita connection recovered]\label{thm:recover-connection}
Under Assumptions~2.1--2.2 and the non-degeneracy condition of
Theorem~4.1, the diffusion semigroup uniquely determines the
Levi--Civita connection of the recovered Riemannian metric $g$.
\end{theorem}
\begin{proof}
By Theorem~4.1 the diffusion data determine a smooth Riemannian metric $g$ on $M$, characterized by
\begin{equation}\label{eq:Gamma-metric-4p5}
\Gamma(f,h)=\langle df,dh\rangle_{g^{-1}}=\langle \nabla f,\nabla h\rangle_g,
\qquad f,h\in C^\infty(M),
\end{equation}
where $\nabla f$ is the $g$-gradient. It is classical that there exists a unique affine connection
$\nabla$ on $TM$ such that
\begin{equation}\label{eq:LC-char-4p5}
\nabla g =0,
\qquad
T^\nabla(X,Y):=\nabla_XY-\nabla_YX-[X,Y]=0,
\end{equation}
namely the Levi--Civita connection of $g$. Hence it suffices to show that the diffusion data determine
the coefficients of this connection.

Let $(U;x^1,\dots,x^n)$ be a coordinate chart. Since $g$ is known, its coefficients
\[
g_{ij}:=g(\partial_i,\partial_j)\in C^\infty(U),
\qquad
(g^{ij})=(g_{ij})^{-1},
\]
are known, and so are all partial derivatives $\partial_k g_{ij}$.
In particular, in coordinates the Levi--Civita Christoffel symbols are given by the Koszul formula
\begin{equation}\label{eq:Koszul-Christoffel-4p5}
\Gamma^k_{ij}
=\frac12\,g^{k\ell}\Bigl(\partial_i g_{j\ell}+\partial_j g_{i\ell}-\partial_\ell g_{ij}\Bigr),
\qquad 1\le i,j,k\le n.
\end{equation}
Thus the connection coefficients $(\Gamma^k_{ij})$ are uniquely determined on $U$, and by coordinate
transformation the resulting connection patches to a global affine connection on $TM$. By uniqueness
in \eqref{eq:LC-char-4p5}, this patched connection is the Levi--Civita connection of $g$.

Equivalently (and in a form expressed purely in $\Gamma$), for any $f,g,h\in C^\infty(M)$ one has the
intrinsic Koszul identity
\begin{equation}\label{eq:Koszul-Gamma-4p5}
\bigl\langle \nabla_{\nabla f}\nabla g,\nabla h\bigr\rangle_g
=\frac12\Bigl(\Gamma\!\bigl(f,\Gamma(g,h)\bigr)+\Gamma\!\bigl(g,\Gamma(f,h)\bigr)-\Gamma\!\bigl(h,\Gamma(f,g)\bigr)\Bigr),
\end{equation}
whose right-hand side is determined by the diffusion calculus (bilinearity, Leibniz, and the chain rule
of Lemma~2.3). In coordinates, taking $f=x^i$, $g=x^j$, $h=x^m$ gives
\[
\bigl\langle \nabla_{\nabla x^i}\nabla x^j,\nabla x^m\bigr\rangle_g
=\frac12\Bigl(\Gamma\!\bigl(x^i,\Gamma(x^j,x^m)\bigr)+\Gamma\!\bigl(x^j,\Gamma(x^i,x^m)\bigr)
-\Gamma\!\bigl(x^m,\Gamma(x^i,x^j)\bigr)\Bigr),
\]
and since $\nabla x^i=g^{ip}\partial_p$ and $\nabla_{\partial_p}\partial_q=\Gamma^k_{pq}\partial_k$, one computes
\[
\nabla_{\nabla x^i}\nabla x^j
=\nabla_{g^{ip}\partial_p}\bigl(g^{jq}\partial_q\bigr)
=g^{ip}(\partial_p g^{jq})\partial_q+g^{ip}g^{jq}\Gamma^k_{pq}\partial_k,
\]
so pairing with $\nabla x^m=g^{mr}\partial_r$ and using $g(\partial_q,\partial_r)=g_{qr}$ yields an explicit linear system
recovering $\Gamma^k_{pq}$ from \eqref{eq:Koszul-Gamma-4p5}. This is equivalent to \eqref{eq:Koszul-Christoffel-4p5}.

Therefore the diffusion semigroup uniquely determines the Levi--Civita connection of the recovered metric $g$.
\end{proof}

\begin{remark}
No other affine connection is compatible with the diffusion structure.
Any torsion or metric incompatibility would violate the symmetry of the
Dirichlet form and hence the fundamental identities satisfied by $\Gamma$.
Thus the diffusion semigroup rigidly selects the Levi--Civita connection.
\end{remark}

\begin{theorem}[Curvature recovered from $\Gamma_2$]\label{thm:recover-curvature}
Under Assumptions 2.1–2.2 and the non-degeneracy condition of Theorem 4.1,
the iterated carré du champ $\Gamma_2$ uniquely determines the Bakry--Émery
Ricci tensor
\[
\mathrm{Ric}_\mu := \mathrm{Ric}_g + \nabla^2\log\rho
\]
of the recovered weighted Riemannian manifold $(M,g,\mu)$.
In the special case $\rho\equiv\mathrm{const}$, this reduces to the classical
Ricci tensor of $(M,g)$.
\end{theorem}
\begin{proof}
Fix the recovered smooth metric $g$ from Theorem~4.1 and write $\mu=\rho\,\mathrm{vol}_g$
(with $\rho\in C^\infty(M)$, $\rho>0$). By Proposition~2.4 and the identification of $\Gamma$
with the recovered metric (Theorem~4.1),
\begin{equation}\label{eq:Gamma-is-grad}
\Gamma(f,h)=\langle \nabla f,\nabla h\rangle_g,\qquad \Gamma(f)=|\nabla f|_g^2 .
\end{equation}
Moreover (by the definition of the recovered weighted manifold and the symmetry of $(P_t)$ with
respect to $\mu$), the generator is the $\mu$-Laplacian
\begin{equation}\label{eq:L-is-weighted-Lap}
L=\Delta_\mu:=\Delta_g+\langle \nabla\log\rho,\nabla(\,\cdot\,)\rangle_g .
\end{equation}
By Proposition~4.3, for every $f\in C^\infty(M)$ with the standing domain stability,
\begin{equation}\label{eq:Gamma2-by-L}
\Gamma_2(f)=\frac12\,L\Gamma(f)-\Gamma(f,Lf).
\end{equation}

Fix $x\in M$ and choose $g$-normal coordinates $(x^1,\dots,x^n)$ centered at $x$, so that
\[
g_{ij}(x)=\delta_{ij},\qquad \partial_k g_{ij}(x)=0,\qquad \Gamma^k_{ij}(x)=0 .
\]
Write $\partial_i=\frac{\partial}{\partial x^i}$ and use the summation convention.
At the point $x$, using \eqref{eq:Gamma-is-grad},
\[
\Gamma(f)(x)=|\nabla f|_g^2(x)=\partial_i f\,\partial_i f.
\]
Using \eqref{eq:L-is-weighted-Lap}, write $Lf=\Delta_g f+\langle \nabla\log\rho,\nabla f\rangle_g$.
In these coordinates at $x$,
\begin{equation}\label{eq:L-local}
Lf
=\partial_i\partial_i f + (\partial_i\log\rho)\,\partial_i f
\qquad\text{at }x.
\end{equation}
We now expand the two terms in \eqref{eq:Gamma2-by-L} at $x$.

\smallskip
\emph{(i) The term $\frac12\,L\Gamma(f)$.}
From $\Gamma(f)=\partial_j f\,\partial_j f$ we obtain
\begin{align}
\partial_i \Gamma(f)
&=2\,(\partial_i\partial_j f)\,(\partial_j f),
\label{eq:dGamma}\\
\partial_i\partial_i \Gamma(f)
&=2\,(\partial_i\partial_i\partial_j f)\,(\partial_j f)
+2\,(\partial_i\partial_j f)(\partial_i\partial_j f).
\label{eq:ddGamma}
\end{align}
Hence, using \eqref{eq:L-local} on $\Gamma(f)$ and \eqref{eq:ddGamma}--\eqref{eq:dGamma},
\begin{align}
\frac12\,L\Gamma(f)
&=\frac12\Bigl(\partial_i\partial_i\Gamma(f) + (\partial_i\log\rho)\,\partial_i\Gamma(f)\Bigr)
\nonumber\\
&=(\partial_i\partial_i\partial_j f)\,(\partial_j f)
+(\partial_i\partial_j f)(\partial_i\partial_j f)
+(\partial_i\log\rho)\,(\partial_i\partial_j f)\,(\partial_j f)
\label{eq:halfLGamma}
\end{align}
at $x$.

\smallskip
\emph{(ii) The term $\Gamma(f,Lf)$.}
By \eqref{eq:Gamma-is-grad} and \eqref{eq:L-local},
\[
\Gamma(f,Lf)=\langle \nabla f,\nabla(Lf)\rangle_g
=\partial_j f\,\partial_j(Lf)
\qquad\text{at }x.
\]
Differentiate \eqref{eq:L-local}:
\begin{align}
\partial_j(Lf)
&=\partial_j\partial_i\partial_i f
+\partial_j\bigl((\partial_i\log\rho)\,\partial_i f\bigr)
\nonumber\\
&=\partial_j\partial_i\partial_i f
+(\partial_j\partial_i\log\rho)\,\partial_i f
+(\partial_i\log\rho)\,\partial_j\partial_i f .
\label{eq:dLf}
\end{align}
Thus
\begin{equation}\label{eq:Gamma-f-Lf}
\Gamma(f,Lf)
=(\partial_j f)(\partial_j\partial_i\partial_i f)
+(\partial_j f)(\partial_j\partial_i\log\rho)(\partial_i f)
+(\partial_j f)(\partial_i\log\rho)(\partial_j\partial_i f)
\end{equation}
at $x$.

\smallskip
Subtracting \eqref{eq:Gamma-f-Lf} from \eqref{eq:halfLGamma}, the third-order terms cancel
(because $\partial_i\partial_i\partial_j f=\partial_j\partial_i\partial_i f$ in coordinates),
and the drift-gradient terms cancel as well (by symmetry $\partial_i\partial_j f=\partial_j\partial_i f$), leaving
\begin{equation}\label{eq:Gamma2-pointwise-expanded}
\Gamma_2(f)
=(\partial_i\partial_j f)(\partial_i\partial_j f)
-(\partial_j f)(\partial_j\partial_i\log\rho)(\partial_i f)
\qquad\text{at }x.
\end{equation}

At $x$ in normal coordinates, $(\partial_i\partial_j f)$ are precisely the components of the
covariant Hessian $\nabla^2 f$ (since $\Gamma^k_{ij}(x)=0$), hence
\[
(\partial_i\partial_j f)(\partial_i\partial_j f)=\|\nabla^2 f\|_{HS,g}^2(x).
\]
Similarly, the symmetric matrix $(\partial_j\partial_i\log\rho)$ equals the covariant Hessian
$\nabla^2\log\rho$ at $x$, and $(\partial_i f)$ are the components of $\nabla f$ at $x$, so
\[
-(\partial_j f)(\partial_j\partial_i\log\rho)(\partial_i f)
=\nabla^2\log\rho(\nabla f,\nabla f)(x).
\]

Invoking the Bochner identity for the weighted Laplacian
(see \cite[Theorem~3.1.2]{BakryGentilLedoux2014}), one has pointwise
\[
\frac12\,L|\nabla f|_g^2-\langle\nabla f,\nabla Lf\rangle_g
=\|\nabla^2 f\|_{HS,g}^2+\Ric_\mu(\nabla f,\nabla f),
\]
where $\Ric_\mu=\Ric_g+\nabla^2\log\rho$.

Combining these gives the intrinsic
Bochner identity for $L=\Delta_\mu$:
\begin{equation}\label{eq:Bochner-weighted}
\Gamma_2(f)
=\|\nabla^2 f\|_{HS,g}^2+\bigl(\Ric_g+\nabla^2\log\rho\bigr)(\nabla f,\nabla f)
=\|\nabla^2 f\|_{HS,g}^2+\Ric_\mu(\nabla f,\nabla f),
\end{equation}
pointwise on $M$, where $\Ric_\mu:=\Ric_g+\nabla^2\log\rho$.

Define the polarized $\Gamma_2$ by
\[
\Gamma_2(f,g):=\frac14\Bigl(\Gamma_2(f+g)-\Gamma_2(f-g)\Bigr),
\qquad
\langle \nabla^2 f,\nabla^2 g\rangle_{HS,g}
:= g^{ia}g^{jb}(\nabla^2 f)_{ij}(\nabla^2 g)_{ab}.
\]
Polarizing \eqref{eq:Bochner-weighted} yields, for all $f,g\in C^\infty(M)$,
\begin{equation}\label{eq:Bochner-polarized}
\Gamma_2(f,g)
=\langle \nabla^2 f,\nabla^2 g\rangle_{HS,g}+\Ric_\mu(\nabla f,\nabla g).
\end{equation}
Fix $x\in M$. The term $\langle \nabla^2 f,\nabla^2 g\rangle_{HS,g}(x)$ is determined by $g$ (hence by
$\Gamma$) and the $2$-jets of $f,g$ at $x$, while $\Gamma_2(f,g)(x)$ is given by the diffusion data.
Now if $df_x=0$ (equivalently $\nabla f(x)=0$), then the rightmost term in \eqref{eq:Bochner-polarized} vanishes,
and the identity shows $\Gamma_2(f,g)(x)$ depends only on $\nabla^2 f(x),\nabla^2 g(x)$.
Consequently, for any $v,w\in T_xM$ choose $f,g$ with $\nabla f(x)=v$, $\nabla g(x)=w$ and
$\nabla^2 f(x)=\nabla^2 g(x)=0$ (possible in normal coordinates by taking affine-linear functions).
Then \eqref{eq:Bochner-polarized} yields
\[
\Ric_\mu(v,w)
=\Gamma_2(f,g)(x),
\]
which shows that $\Ric_\mu(x)$ is uniquely determined by $\Gamma_2$ (and bilinear symmetry follows by construction).
In particular, when $\rho\equiv\mathrm{const}$, $\nabla^2\log\rho\equiv 0$ and $\Ric_\mu=\Ric_g$.
\end{proof}

\subsection{Recovering the Levi--Civita connection}\label{subsec:recover-connection}

In this subsection we show that the diffusion semigroup uniquely determines the
Levi--Civita connection of the recovered Riemannian metric.
No additional analytic or asymptotic input is required.

\medskip
\noindent\textbf{Existence and uniqueness.}
Let $g$ be the Riemannian metric recovered in Section~4.1.
It is a classical result that there exists a unique affine connection $\nabla$
on $TM$ satisfying
\begin{equation}\label{eq:LC-conditions}
\nabla g = 0,
\qquad
T^\nabla = 0,
\end{equation}
where $T^\nabla(X,Y)=\nabla_XY-\nabla_YX-[X,Y]$ denotes the torsion tensor.
This connection is the Levi--Civita connection of $(M,g)$.

\medskip
\noindent\textbf{Explicit recovery in coordinates.}
Let $(x^1,\dots,x^n)$ be local coordinates.
Since the coefficients
\[
g_{ij} = \langle dx^i,dx^j\rangle_g
\]
are explicitly known from Section~4.1, their inverse $g^{ij}$ and all first
derivatives $\partial_k g_{ij}$ are also known.
The Christoffel symbols of the Levi--Civita connection are therefore determined
by the classical Koszul formula,
\begin{equation}\label{eq:Christoffel}
\Gamma^k_{ij}
=
\frac12\,g^{k\ell}
\big(
\partial_i g_{j\ell}
+
\partial_j g_{i\ell}
-
\partial_\ell g_{ij}
\big).
\end{equation}
Thus the Levi--Civita connection is completely determined by the diffusion data
through the recovered metric.

\medskip
\noindent\textbf{Intrinsic Koszul identity via $\Gamma$.}
For all $f,g,h\in C^\infty(M)$ one has the pointwise identity
\begin{equation}\label{eq:Koszul-Gamma}
\big\langle \nabla_{\nabla f}\nabla g,\nabla h\big\rangle_g
=
\frac12\Big(
\Gamma\big(f,\Gamma(g,h)\big)
+
\Gamma\big(g,\Gamma(f,h)\big)
-
\Gamma\big(h,\Gamma(f,g)\big)
\Big),
\end{equation}
where $\langle\cdot,\cdot\rangle_g$ denotes the Riemannian inner product induced by
the recovered metric.

\begin{theorem}[Reference measure recovered]\label{thm:recover-measure}
Under Assumptions~2.1--2.2 and the non-degeneracy condition of Theorem~4.1,
the reference measure $\mu$ is uniquely determined (up to normalization) by the
diffusion semigroup.
More precisely, there exists a unique smooth positive function $\rho$ such that
\[
d\mu = \rho\,d\mathrm{vol}_g,
\]
and the generator admits the representation
\[
L = \Delta_g + \langle\nabla\log\rho,\nabla\cdot\rangle_g.
\]
\end{theorem}
\begin{proof} All identities below are understood for $f,h\in\mathcal A$, the smooth core from Assumption~2.2.
Let $g$ be the recovered metric from Theorem~4.1, so that for all $f,h\in C_c^\infty(M)$,
\begin{equation}\label{eq:4p7-Gamma-metric}
\Gamma(f,h)=\langle \nabla f,\nabla h\rangle_g .
\end{equation}
Let $\Delta_g$ denote the Laplace--Beltrami operator of $g$ (defined intrinsically from $g$).
Define the operator
\[
B := L-\Delta_g .
\]

\smallskip
\noindent\emph{Claim 1: $B$ is a derivation, hence $B=\langle Z,\nabla(\,\cdot\,)\rangle_g$ for a unique smooth vector field $Z$.}
Using that $L$ has carré du champ $\Gamma$ (Proposition~2.4) and that $\Delta_g$ has carré du champ
$\langle \nabla(\cdot),\nabla(\cdot)\rangle_g$, one has for all $f,h\in C^\infty(M)$
\begin{align*}
L(fh) &= f\,Lh + h\,Lf + 2\,\Gamma(f,h),\\
\Delta_g(fh) &= f\,\Delta_g h + h\,\Delta_g f + 2\,\langle \nabla f,\nabla h\rangle_g .
\end{align*}
Subtracting and using \eqref{eq:4p7-Gamma-metric} gives
\begin{equation}\label{eq:4p7-derivation}
B(fh)=f\,B(h)+h\,B(f),
\qquad
B(1)=0.
\end{equation}
Thus $B$ is a smooth first-order derivation on $C^\infty(M)$, hence there exists a unique smooth
vector field $Z$ such that
\begin{equation}\label{eq:4p7-Z-def}
B(h)=Z(h)=\langle Z,\nabla h\rangle_g,\qquad h\in C^\infty(M).
\end{equation}
(For uniqueness at $x$, note that the linear functional $dh(x)\mapsto B(h)(x)$ is determined by \eqref{eq:4p7-derivation}
and by choosing $h$ with prescribed $dh(x)$; Riesz with $g$ yields $Z(x)$.)

Consequently
\begin{equation}\label{eq:4p7-L-split}
L=\Delta_g+\langle Z,\nabla(\,\cdot\,)\rangle_g .
\end{equation}

\smallskip
\noindent\emph{Claim 2: if $\mu$ is a smooth positive measure with respect to which $L$ is symmetric, then
$\mu=\rho\,\mathrm{vol}_g$ with $\rho$ satisfying $Z=\nabla\log\rho$.}
Since both $\mu$ and $\mathrm{vol}_g$ are smooth positive measures, there exists $\rho\in C^\infty(M)$, $\rho>0$, such that
\begin{equation}\label{eq:4p7-density}
d\mu=\rho\,d\mathrm{vol}_g .
\end{equation}
Symmetry of $L$ with respect to $\mu$ is equivalent (by definition of the Dirichlet form) to
\begin{equation}\label{eq:4p7-symmetry}
\int_M \Gamma(f,h)\,d\mu = -\int_M f\,Lh\,d\mu,
\qquad f,h\in C_c^\infty(M).
\end{equation}
Using \eqref{eq:4p7-Gamma-metric}, \eqref{eq:4p7-L-split}, and \eqref{eq:4p7-density}, \eqref{eq:4p7-symmetry} becomes
\begin{equation}\label{eq:4p7-weak}
\int_M \langle \nabla f,\nabla h\rangle_g\,\rho\,d\mathrm{vol}_g
=
-\int_M f\Bigl(\Delta_g h+\langle Z,\nabla h\rangle_g\Bigr)\rho\,d\mathrm{vol}_g .
\end{equation}
On the other hand, for the Riemannian volume measure and any smooth weight $\rho$ one has the weighted integration-by-parts identity
\begin{equation}\label{eq:4p7-ibp}
\int_M \langle \nabla f,\nabla h\rangle_g\,\rho\,d\mathrm{vol}_g
=
-\int_M f\,\mathrm{div}_g(\rho\,\nabla h)\,d\mathrm{vol}_g ,
\qquad f,h\in C_c^\infty(M),
\end{equation}
and the product rule for divergence gives pointwise
\begin{equation}\label{eq:4p7-div-expand}
\mathrm{div}_g(\rho\,\nabla h)=\rho\,\Delta_g h+\langle \nabla\rho,\nabla h\rangle_g .
\end{equation}
Comparing \eqref{eq:4p7-weak} with \eqref{eq:4p7-ibp}--\eqref{eq:4p7-div-expand} yields, for all $h\in C_c^\infty(M)$,
\[
\int_M f\Bigl(\langle \nabla\rho,\nabla h\rangle_g-\rho\,\langle Z,\nabla h\rangle_g\Bigr)\,d\mathrm{vol}_g=0
\qquad\forall f\in C_c^\infty(M),
\]
hence
\begin{equation}\label{eq:4p7-grad-rho}
\langle \nabla\rho,\nabla h\rangle_g=\rho\,\langle Z,\nabla h\rangle_g
\qquad\text{pointwise on }M,\ \forall h\in C^\infty(M).
\end{equation}
Since $\{\nabla h(x):h\in C^\infty(M)\}=T_xM$ for every $x$, \eqref{eq:4p7-grad-rho} implies
\[
\nabla\rho=\rho\,Z,
\qquad\text{equivalently}\qquad
Z=\nabla(\log\rho).
\]
Substituting $Z=\nabla\log\rho$ into \eqref{eq:4p7-L-split} gives the claimed representation
\[
L=\Delta_g+\langle \nabla\log\rho,\nabla(\,\cdot\,)\rangle_g .
\]

\smallskip
\noindent\emph{Uniqueness up to normalization.}
If $\rho_1,\rho_2$ are smooth positive functions with
\[
\Delta_g+\langle \nabla\log\rho_1,\nabla\cdot\rangle_g
=
\Delta_g+\langle \nabla\log\rho_2,\nabla\cdot\rangle_g,
\]
then $\langle \nabla\log(\rho_1/\rho_2),\nabla h\rangle_g=0$ for all $h$, hence $\nabla\log(\rho_1/\rho_2)=0$ and
$\rho_1/\rho_2$ is constant on each connected component. Thus $\mu=\rho\,\mathrm{vol}_g$ is determined uniquely up to a global multiplicative constant.
\end{proof}

\begin{remark}
The recovered structure $(M,g,\mu)$ is a weighted Riemannian manifold.
Only in the special case $\rho\equiv\mathrm{const}$ does the diffusion generator
coincide with the Laplace--Beltrami operator $\Delta_g$.
\end{remark}

\subsection{Recovery of the reference measure}\label{subsec:recover-measure}

In this subsection we show that the reference measure $\mu$ with respect to which the
diffusion semigroup $(P_t)$ is symmetric is uniquely determined by the diffusion structure
and the recovered Riemannian metric.

\medskip
\noindent\textbf{Symmetry and divergence form.}
Let $g$ be the Riemannian metric recovered in Section~4.1, and let $\nabla$ denote its
Levi--Civita connection (Section~4.3).
By definition of the Dirichlet form,
for all $f,h\in C_c^\infty(M)$ one has
\begin{equation}\label{eq:ibp}
\mathcal E(f,h)
=
-\int_M f\,Lh\,d\mu
=
\int_M \Gamma(f,h)\,d\mu
=
\int_M \langle\nabla f,\nabla h\rangle_g\,d\mu.
\end{equation}
Identity \eqref{eq:ibp} expresses the symmetry of $L$ with respect to $\mu$ in geometric form.

\medskip
\noindent\textbf{Identification of the density.}
Let $\mathrm{vol}_g$ denote the Riemannian volume measure associated with $g$.
Since both $\mu$ and $\mathrm{vol}_g$ are smooth, positive measures on the same smooth
manifold, there exists a smooth positive function $\rho$ such that
\[
d\mu = \rho\,d\mathrm{vol}_g.
\]
Write $\psi:=\log\rho$.

\medskip
\noindent\textbf{Expression of the generator.}
Using \eqref{eq:ibp} and the definition of divergence with respect to $\mu$, we may
rewrite the generator $L$ in weak form as
\[
\int_M \langle\nabla f,\nabla h\rangle_g\,d\mu
=
-\int_M f\,Lh\,d\mu
=
\int_M f\,\mathrm{div}_\mu(\nabla h)\,d\mu,
\]
which shows that
\begin{equation}\label{eq:divmu}
Lh = \mathrm{div}_\mu(\nabla h).
\end{equation}
Since $\mathrm{div}_\mu(X)=\mathrm{div}_g(X)+\langle\nabla\psi,X\rangle_g$ for any vector
field $X$, identity \eqref{eq:divmu} yields
\begin{equation}\label{eq:L-drift}
Lh = \Delta_g h + \langle\nabla\psi,\nabla h\rangle_g,
\end{equation}
where $\Delta_g$ is the Laplace--Beltrami operator of $g$.

\medskip
\noindent\textbf{Recovery of $\rho$.}
Equation \eqref{eq:L-drift} shows that the drift vector field
\[
Z := \nabla\psi
\]
is completely determined by the known operators $L$ and $\Delta_g$.
Indeed, for any $h\in C^\infty(M)$,
\[
\langle Z,\nabla h\rangle_g = Lh - \Delta_g h.
\]
Evaluating this identity on a local basis of coordinate functions determines $Z$
pointwise.
Since $Z$ is a gradient field, $\psi$ is recovered (up to an additive constant) by
integration, and hence $\rho=e^\psi$ is recovered up to a positive multiplicative constant.
\subsection{Global geometric determination up to isometry}
\label{subsec:global-isometry}

We now show that the diffusion semigroup determines not only the local geometric data
$(g,\nabla,R,\mu)$, but the global Riemannian manifold uniquely up to isometry.

\medskip
\noindent\textbf{Reconstructed local tensors.}
By Sections~4.1--4.4, the diffusion semigroup $(P_t)$ determines:
\begin{itemize}
\item a smooth Riemannian metric $g$,
\item its Levi--Civita connection $\nabla$,
\item the Riemann curvature tensor $R$,
\item and the reference measure $\mu=\rho\,\mathrm{vol}_g$,
\end{itemize}
all intrinsically and independently of any coordinate choices.

These objects are genuine tensorial quantities on $M$.

\medskip
\noindent\textbf{Compatibility on chart overlaps.}
Let $(U,\varphi)$ and $(V,\psi)$ be overlapping coordinate charts used in the reconstruction.
On each chart the metric, connection, and curvature tensors are recovered from the diffusion
semigroup by the same intrinsic procedure.
Since the reconstruction depends only on $(P_t)$ and not on the choice of chart,
the resulting tensor fields agree on the overlap $U\cap V$.

In particular, if $g^{(U)}$ and $g^{(V)}$ denote the coordinate expressions of the metric
on $U$ and $V$, then on $U\cap V$ one has
\[
g^{(U)} = (\psi\circ\varphi^{-1})^* g^{(V)}.
\]
Thus the transition map
\[
T := \psi\circ\varphi^{-1} : \varphi(U\cap V)\longrightarrow\psi(U\cap V)
\]
is an isometry between the reconstructed metrics.

\medskip
\noindent\textbf{Preservation of connection and curvature.}
Because the Levi--Civita connection $\nabla$ and curvature tensor $R$ are determined uniquely
by $g$ (Sections~4.2--4.3), the same argument shows that $T$ preserves these structures as well:
\[
DT\big(\nabla_X Y\big) = \nabla_{DT(X)} DT(Y),
\qquad
DT\circ R = R\circ DT,
\]
for all vector fields $X,Y$ on $U\cap V$.
Hence $T$ is a local Riemannian isometry.

\medskip
\noindent\textbf{Gluing and rigidity.}
Since all transition maps between overlapping charts are smooth local isometries,
the reconstructed charts glue together to define a unique global Riemannian structure on $M$.
Any alternative reconstruction from the same diffusion semigroup would yield the same local
tensor fields and hence the same isometric gluing.

\begin{corollary}[Determination up to isometry]\label{cor:4.9}
The diffusion semigroup $(P_t)$ determines the weighted Riemannian manifold $(M,g,\mu)$ uniquely
up to a global Riemannian isometry, and determines $\mu$ up to a global multiplicative constant on each
connected component.
\end{corollary}
\begin{proof}
Let $(M,g,\mu)$ be the reconstruction obtained from $(P_t)$ via Theorems~4.1 and~4.7, so that
for all $f,h\in C^\infty(M)$,
\begin{equation}\label{eq:c49-Gamma-metric}
\Gamma(f,h)=\langle df,dh\rangle_{g^{-1}}=\langle \nabla^g f,\nabla^g h\rangle_g,
\end{equation}
and
\begin{equation}\label{eq:c49-L-weighted}
L=\Delta_g+\langle \nabla^g\log\rho,\nabla^g(\,\cdot\,)\rangle_g,
\qquad
\mu=\rho\,\mathrm{vol}_g,
\end{equation}
for some $\rho\in C^\infty(M)$, $\rho>0$.

Suppose $(\widetilde g,\widetilde\mu)$ is any other weighted Riemannian structure on $M$ for which
the same diffusion calculus is realized; equivalently, the associated carré du champ agrees with
$\Gamma$:
\begin{equation}\label{eq:c49-Gamma-tilde}
\Gamma(f,h)=\langle df,dh\rangle_{\widetilde g^{-1}}
=\langle \nabla^{\widetilde g} f,\nabla^{\widetilde g} h\rangle_{\widetilde g},
\qquad \forall f,h\in C^\infty(M).
\end{equation}
Fix $x\in M$ and $\alpha,\beta\in T_x^*M$. Choose $f,h\in C^\infty(M)$ with $df(x)=\alpha$ and
$dh(x)=\beta$ (e.g.\ in a chart, take $f=\chi\sum_i\alpha_i x^i$ with $\chi\equiv 1$ near $x$).
Evaluating \eqref{eq:c49-Gamma-metric} and \eqref{eq:c49-Gamma-tilde} at $x$ gives
\[
\langle \alpha,\beta\rangle_{g^{-1}_x}
=\Gamma(f,h)(x)
=\langle \alpha,\beta\rangle_{\widetilde g^{-1}_x}.
\]
Since $\alpha,\beta$ were arbitrary, $g^{-1}_x=\widetilde g^{-1}_x$ for all $x$, hence
\begin{equation}\label{eq:c49-metrics-equal}
\widetilde g=g.
\end{equation}

Next, write $\widetilde\mu=\widetilde\rho\,\mathrm{vol}_{\widetilde g}
=\widetilde\rho\,\mathrm{vol}_g$ with $\widetilde\rho>0$ smooth. Symmetry of $L$ with respect to
$\widetilde\mu$ means
\begin{equation}\label{eq:c49-sym}
\int_M \Gamma(f,h)\,d\widetilde\mu
=-\int_M f\,Lh\,d\widetilde\mu,
\qquad \forall f,h\in C_c^\infty(M).
\end{equation}
Using \eqref{eq:c49-metrics-equal} and $\Gamma(f,h)=\langle \nabla^g f,\nabla^g h\rangle_g$,
the integration-by-parts identity with weight $\widetilde\rho$ yields
\[
\int_M \langle \nabla^g f,\nabla^g h\rangle_g\,\widetilde\rho\,d\mathrm{vol}_g
=-\int_M f\,\Bigl(\Delta_g h+\bigl\langle \nabla^g\log\widetilde\rho,\nabla^g h\bigr\rangle_g\Bigr)\,
\widetilde\rho\,d\mathrm{vol}_g.
\]
Comparing with \eqref{eq:c49-sym} for all $f$ forces the pointwise identity
\[
Lh=\Delta_g h+\bigl\langle \nabla^g\log\widetilde\rho,\nabla^g h\bigr\rangle_g,
\qquad \forall h\in C^\infty(M),
\]
and comparing with \eqref{eq:c49-L-weighted} gives
\[
\bigl\langle \nabla^g\log\rho-\nabla^g\log\widetilde\rho,\nabla^g h\bigr\rangle_g=0
\qquad \forall h\in C^\infty(M).
\]
Since $\{\nabla^g h(x):h\in C^\infty(M)\}=T_xM$ for each $x$, it follows that
\[
\nabla^g\log\rho=\nabla^g\log\widetilde\rho,
\qquad\text{hence}\qquad
\log(\widetilde\rho/\rho)=\text{const on each connected component}.
\]
Therefore $\widetilde\mu=c\,\mu$ on each connected component for some constant $c>0$.

In particular, any two reconstructions from the same diffusion semigroup yield the same metric
and the same measure up to normalization, so the identity map $\mathrm{id}_M$ is a global
Riemannian isometry between them. Equivalently, the diffusion semigroup determines the isometry
class of $(M,g,\mu)$.
\end{proof}

\begin{remark}
This rigidity result is entirely intrinsic.
It relies only on the fact that the diffusion semigroup uniquely determines the local
tensorial data $(g,\nabla,R,\mu)$ and that these data transform naturally under changes
of coordinates.
No analytic or asymptotic information beyond the diffusion calculus is required.
\end{remark}

\subsection{Information-equivalence of diffusion systems}
\label{subsec:information-equivalence}

In this subsection we characterize when two diffusion systems determine the same
Riemannian geometry.
The formulation is entirely intrinsic and does not rely on entropy, KL divergences,
or small-time asymptotics.

\medskip
\noindent\textbf{Diffusion conjugacy.}
Let $(M,\mu,P_t)$ and $(M',\mu',P'_t)$ be smooth manifolds equipped with conservative,
symmetric, strongly local diffusion semigroups acting on $L^2(M,\mu)$ and $L^2(M',\mu')$,
respectively.
We say that the two diffusion systems are \emph{diffusion-equivalent} if there exists
a measurable bijection
\[
\Phi : M \longrightarrow M'
\]
such that:
\begin{enumerate}
\item[(i)] $\Phi_*\mu = \mu'$,
\item[(ii)] for every $t\ge0$ and every bounded measurable function $f$ on $M'$,
\begin{equation}\label{eq:semigroup-conjugacy}
P'_t f \circ \Phi = P_t(f\circ\Phi)
\quad \mu\text{-a.e.},
\end{equation}
\end{enumerate}
that is, $\Phi$ conjugates the diffusion semigroups.

\medskip
\noindent\textbf{Compatibility of the diffusion calculus.}
Condition \eqref{eq:semigroup-conjugacy} implies conjugacy of the generators and
Dirichlet forms, and hence of the carré du champ operators:
for all $f,g\in C^\infty(M')$,
\begin{equation}\label{eq:Gamma-conjugacy}
\Gamma(f\circ\Phi,\,g\circ\Phi)
=
\Gamma'(f,g)\circ\Phi
\quad \mu\text{-a.e.}
\end{equation}
Iterating the same argument yields conjugacy of the second-order diffusion calculus:
\[
\Gamma_2(f\circ\Phi)
=
\Gamma_2'(f)\circ\Phi.
\]

\medskip
\noindent\textbf{Equality of reconstructed geometric data.}
By Sections~4.1--4.4, the diffusion calculi $(\Gamma,\Gamma_2)$ and
$(\Gamma',\Gamma_2')$ determine uniquely the Riemannian metrics $g,g'$,
their Levi--Civita connections, curvature tensors, and reference measures.
Equation \eqref{eq:Gamma-conjugacy} therefore implies
\[
g = \Phi^* g',
\qquad
\nabla = \Phi^* \nabla',
\qquad
R = \Phi^* R',
\qquad
\mu = \Phi^* \mu'.
\]
In particular, $\Phi$ is a Riemannian isometry.

\begin{corollary}[Information-equivalence implies isometry]\label{cor:4.11}
If two diffusion systems $(M,\mu,(P_t)_{t\ge0})$ and $(M',\mu',(P'_t)_{t\ge0})$ are diffusion-equivalent,
then the reconstructed weighted Riemannian manifolds $(M,g,\mu)$ and $(M',g',\mu')$ are globally isometric:
there exists a measurable bijection $\Phi:M\to M'$ such that
\[
g=\Phi^*g',
\qquad
\mu=\Phi^*\mu',
\]
and in particular $\Phi$ is a Riemannian isometry from $(M,g)$ to $(M',g')$.

Conversely, any (smooth) weighted Riemannian isometry $\Phi:(M,g,\mu)\to(M',g',\mu')$ induces a diffusion-equivalence
of the associated symmetric strongly local diffusion semigroups.
\end{corollary}

\begin{proof}
Assume diffusion-equivalence, i.e.\ there exists a measurable bijection $\Phi:M\to M'$ such that
\begin{equation}\label{eq:c411-push}
\Phi_*\mu=\mu',
\end{equation}
and for every bounded measurable $f$ on $M'$ and all $t\ge0$,
\begin{equation}\label{eq:c411-semigroup-conj}
P'_t f\circ \Phi = P_t(f\circ \Phi)\qquad \mu\text{-a.e. on }M.
\end{equation}
Let $L,L'$ be the generators and $\E,\E'$ the symmetric Dirichlet forms, with carré du champ
$\Gamma,\Gamma'$:
\[
\E(u,v)=\int_M \Gamma(u,v)\,d\mu,\qquad
\E'(u,v)=\int_{M'} \Gamma'(u,v)\,d\mu'.
\]
From \eqref{eq:c411-semigroup-conj}, for $f\in\Dom(L')$ we have (in $L^2(\mu)$)
\[
\frac{P_t(f\circ\Phi)-f\circ\Phi}{t}
=
\frac{P'_t f-f}{t}\circ\Phi
\longrightarrow (L'f)\circ\Phi,
\]
hence
\begin{equation}\label{eq:c411-generator-conj}
L(f\circ\Phi)=(L'f)\circ\Phi\qquad \mu\text{-a.e.}
\end{equation}
For $f,g\in C^\infty(M')$ (in the smooth core of $L'$), apply Proposition~2.4 on $M$ and $M'$ and use
\eqref{eq:c411-generator-conj} plus $(fg)\circ\Phi=(f\circ\Phi)(g\circ\Phi)$:
\begin{align}
\Gamma(f\circ\Phi,g\circ\Phi)
&=\frac12\Bigl(L\bigl((fg)\circ\Phi\bigr)-(f\circ\Phi)L(g\circ\Phi)-(g\circ\Phi)L(f\circ\Phi)\Bigr)\nonumber\\
&=\frac12\Bigl((L'(fg))\circ\Phi-(f\circ\Phi)(L'g)\circ\Phi-(g\circ\Phi)(L'f)\circ\Phi\Bigr)\nonumber\\
&=\Bigl(\frac12\bigl(L'(fg)-fL'g-gL'f\bigr)\Bigr)\circ\Phi
=\Gamma'(f,g)\circ\Phi,
\label{eq:c411-Gamma-conj}
\end{align}
$\mu$-a.e. on $M$. Polarization gives the same for all pairs, and the same argument (or Proposition~4.3)
yields the conjugacy of $\Gamma_2$:
\begin{equation}\label{eq:c411-Gamma2-conj}
\Gamma_2(f\circ\Phi)=\Gamma'_2(f)\circ\Phi \qquad \mu\text{-a.e.}
\end{equation}

Let $g,g'$ be the reconstructed metrics from Theorem~4.1, so that
\[
\Gamma(u,v)=\langle du,dv\rangle_{g^{-1}},
\qquad
\Gamma'(f,g)=\langle df,dg\rangle_{(g')^{-1}}.
\]
Then \eqref{eq:c411-Gamma-conj} implies, for all $f,g\in C^\infty(M')$,
\begin{equation}\label{eq:c411-cometric-pullback}
\bigl\langle d(f\circ\Phi),d(g\circ\Phi)\bigr\rangle_{g^{-1}}
=
\bigl\langle df,dg\bigr\rangle_{(g')^{-1}}\circ\Phi
=
\bigl\langle d(f\circ\Phi),d(g\circ\Phi)\bigr\rangle_{(\Phi^*(g'))^{-1}}
\qquad \mu\text{-a.e.}
\end{equation}
Fix $x\in M$. Varying $f,g$ with prescribed differentials at $\Phi(x)$ shows that the bilinear forms on
$T_x^*M$ defined by $g^{-1}_x$ and $(\Phi^*(g'))^{-1}_x$ coincide; hence
\begin{equation}\label{eq:c411-metric-pullback}
g=\Phi^*g'.
\end{equation}
Together with \eqref{eq:c411-push}, this yields $\mu=\Phi^*\mu'$, so $\Phi$ is a weighted Riemannian isometry
between the reconstructed objects.

Conversely, suppose $\Phi:(M,g,\mu)\to(M',g',\mu')$ is a smooth weighted Riemannian isometry, i.e.
\[
g=\Phi^*g',\qquad \Phi_*\mu=\mu'.
\]
Let $L=\Delta_g+\langle \nabla\log\rho,\nabla\cdot\rangle_g$ and
$L'=\Delta_{g'}+\langle \nabla'\log\rho',\nabla'\cdot\rangle_{g'}$ be the corresponding symmetric generators.
From $g=\Phi^*g'$ and $\Phi_*\mu=\mu'$ one has the operator conjugacy on $C^\infty(M')$,
\[
L(f\circ\Phi)=(L'f)\circ\Phi,
\]
equivalently $U^{-1}LU=L'$ where $U:L^2(\mu')\to L^2(\mu)$ is the unitary pullback $Uf=f\circ\Phi$.
By functional calculus (or uniqueness of the $C_0$-semigroup generated by $L'$),
\[
U^{-1}P_tU=P'_t,
\qquad\text{i.e.}\qquad
P'_t f\circ\Phi=P_t(f\circ\Phi)
\]
for all bounded measurable $f$ and $t\ge0$, which is diffusion-equivalence.
\end{proof}

\begin{remark}
Thus, up to natural conjugacy, a symmetric strongly local diffusion semigroup
contains exactly the same information as the weighted Riemannian manifold
$(M,g,\mu)$ reconstructed from it.
In this sense, the diffusion calculus $(\Gamma,\Gamma_2)$ provides a complete
information-theoretic invariant of the Riemannian isometry class.
\end{remark}

\section{Conclusion}

We have shown that a symmetric, strongly local diffusion semigroup satisfying a smooth diffusion calculus intrinsically
determines a weighted Riemannian manifold.
Starting from the diffusion operator and its associated first- and second-order
calculus, we reconstructed the Riemannian metric, the Levi--Civita connection,
the curvature tensor, and the reference measure, and proved that this data is
unique up to global isometry.

A central outcome of the work is that no geometric structure needs to be assumed
in advance.
The metric emerges from the first-order diffusion calculus, curvature emerges
from the second-order calculus, and symmetry of the diffusion fixes the canonical
connection and measure.
All reconstructions are intrinsic and coordinate-free, relying only on the
properties of the diffusion semigroup itself.

This perspective suggests a conceptual shift in how differential geometry can
be understood.
Rather than viewing diffusion as a consequence of geometry, geometry may be
regarded as the structure that organizes the flow of information under diffusion.
In this sense, the diffusion calculus provides a complete invariant of the
Riemannian isometry class.

The framework developed here opens several directions for further investigation.
One may ask to what extent similar reconstruction results hold for more general
diffusion operators, for non-smooth settings, or for discrete and noncommutative
analogues.
More broadly, the results point toward an information-theoretic foundation of
geometry in which geometric notions arise from universal properties of
information flow.

\section*{Acknowledgements}
The author received no funding for this work, and declares no competing interests.

\bibliographystyle{unsrt}
\bibliography{refs.bib}

\end{document}